\documentclass[11pt,a4paper]{amsart}

\usepackage[backend=biber,style=alphabetic,maxbibnames=50]{biblatex}
\usepackage{color}
\usepackage{xcolor}

\usepackage[utf8]{inputenc}
\usepackage[T1]{fontenc} 
\usepackage{geometry} 
\usepackage[french, english]{babel}
\usepackage{amsmath}
\usepackage{amsfonts}
\usepackage{amsthm}

\usepackage{amscd}
\usepackage{hyperref}
\usepackage{xcolor}
\usepackage{enumitem}

\usepackage{graphicx}

\usepackage{pstricks,pst-plot}

   \usepackage{amssymb}
   \usepackage{verbatim}
   \newif\ifpdf
   \ifpdf
     \usepackage[pdftex]{graphicx}
     \usepackage[pdftex]{hyperref}
   \else
     \usepackage{graphicx}
   \fi

\usepackage[all]{xy}

\newcommand{\R}{\mathbb{R}}

\newcommand{\N}{\mathbb{N}}

\newcommand{\C}{\mathbb{C}}

\newtheorem{thm}{Theorem}[section]
\newtheorem{prop}[thm]{Proposition}
\newtheorem{coro}[thm]{Corollary}
\newtheorem{defi}[thm]{Definition}
\newtheorem{lem}[thm]{Lemma}

\newtheorem{rem}[thm]{Remark}

\newcommand{\limn}{\lim_{n \rightarrow \infty}}

\newcommand{\bcal}{\mathcal{B}}
\newcommand{\ptwo}{\mathbb{P}^2}

\newcommand{\re}{\mathrm{Re}}
\newcommand{\im}{\mathrm{Im}}

\newcommand{\eps}{\varepsilon}

\author{Matthieu Astorg}
\address{Université d’Orléans, Institut Denis
Poisson, UMR CNRS 7013, 45067 Orléans Cedex 2, France}
\email{matthieu.astorg@univ-orleans.fr}
\author{Fabrizio Bianchi}
\address{Dipartimento di Matematica, Università di Pisa, Largo Bruno Pontecorvo 5, 56127 Pisa, Italy}
 \email{fabrizio.bianchi$@$unipi.it}

\title{Non-autonomous parabolic implosion}

\begin{document}

\maketitle
\selectlanguage{english}

\begin{abstract}
We study parabolic implosion in a 
general
non-autonomous setting. 
Let $f(w)=w+w^2+O(w^3)$ be a holomorphic germ tangent to the identity.
We consider
the iteration of  non-autonomous perturbations  of the form
\[
w_{j+1}=f(w_j)+\varepsilon_{j,n}^2.
\]
We show that, when the 
$\varepsilon_{j,n}^2$'s satisfy a 
Lavaurs-type condition, the element 
$w_n$
 can be described by means of a suitable Lavaurs map
$L_{u_n}$,
whose 
phase $u_n$ 
is an explicit function 
of the perturbation parameters. In particular, whenever $u_n\to u\in \mathbb C$,
the non-autonomous dynamics converges locally uniformly on compact subsets of the parabolic basin
to the corresponding Lavaurs map $L_u$.

Our study
 provides a general description of additive non-autonomous parabolic
implosion and yields several
 deterministic and random convergence results as
corollaries, as well as 
	a unified proof of several previous results.
	As an application, we also
	obtain strong discontinuity results
	for the Julia sets of fibered holomorphic endomorphisms of $\mathbb P^2(\mathbb C)$.
\end{abstract}

\section{Introduction}

\subsection{Parabolic points and implosion}
Parabolic fixed points are among the main sources of bifurcations in one--dimensional holomorphic
dynamics; we refer to \cite{milnor2011dynamics,carleson2013complex}
 for background.  Let
\[
f(w)=w+w^{2}+O(w^{3})
\]
be a holomorphic germ tangent to the identity at the origin.  The local dynamics of $f$ is described
by attracting and repelling Fatou coordinates and the associated horn maps, which also encode the
analytic classification of simple parabolic germs (\'Ecalle--Voronin theory
\cite{Ecalle85-tome3,Voronin81analytic}),
 see e.g.
\cite{DudkoSauzin2014, DudkoSauzin2015}.

In the classical \emph{autonomous} setting, parabolic implosion concerns the limiting dynamics of
suitably rescaled perturbations of $f$.  Following the seminal work of Douady--Hubbard and Lavaurs
\cite{notesorsay1, notesorsay2, lavaurs1989systemes}, one considers additive perturbations
\[
f_{\varepsilon}(w)=f(w)+\varepsilon^{2},
\]
with $\varepsilon\to 0$ along sequences of the form
\[
\varepsilon_n=\frac{\pi}{n}+\frac{\pi \sigma}{n^{2}}+O\Big(\frac{1}{n^{2+\alpha}}\Big),\qquad \alpha>0.
\]
Then the renormalized dynamics after $n$ iterates converges (locally uniformly on compact subsets of
the parabolic basin $\mathcal B_f$) to a \emph{Lavaurs map} $L_{\sigma}$ associated with $f$.
Parabolic implosion has deep
 consequences in complex dynamics; for instance it explains the
discontinuity of Julia sets at parabolic parameters \cite{Douady94Julia} and it plays a key role in
Shishikura's proof that the boundary of the Mandelbrot set has Hausdorff dimension two \cite{shishikura1998hausdorff}.
We refer to \cite{peters2020parabolic} for an overview of the theory and
to \cite{BC12area, cheraghi2015satellite, inou2006renormalization} for further
 consequences and applications.

Parabolic implosion techniques
have also been developed in higher dimension, where these phenomena
 appear
naturally in the study of semi-parabolic or tangent-to-the-identity dynamics.  In particular, they
have been used to construct genuinely higher--dimensional phenomena such as wandering Fatou
components for polynomial skew products \cite{astorg2014two, astorg2019wandering, astorg2022dynamics},
 and to obtain discontinuity statements for
two--dimensional families 
\cite{bedford2017semi, dujardin2015stability, bianchi2019parabolic, astorg2026parabolic}
and structural properties for the chaotic sets
 \cite{astorg2025horn}.

The purpose of this paper is to study an analogue of parabolic implosion in a general
\emph{non-autonomous} setting, where the perturbation is allowed to vary at each iterate.  This type
of problem is motivated both by higher--dimensional skew-product dynamics and by random
models.
  In one complex dimension, a first
result in this direction is due to Vivas
\cite{vivas2020non},
 see also the recent work
\cite{HSV26}.
Here we develop a general additive non-autonomous theory
for arbitrary parabolic germs and obtain an explicit description of the limiting Lavaurs phase.

\subsection{Non-autonomous perturbations and main result}

Let
\begin{equation}\label{eq:int:f}
f(w)=w+w^2+O(w^3)
\end{equation}
be a holomorphic germ at the origin. For each integer $n$, we consider sequences
of perturbations $\varepsilon_{k,n}$ and define the non-autonomous iteration
\begin{equation}\label{eq:int:wj}
w_{k+1}^{(n)}=f(w_k^{(n)})+\varepsilon_{k,n}^2,
\qquad k=0,\dots,n-1.
\end{equation}

We are interested in the asymptotic behavior of $w_n^{(n)}$
 as $n\to\infty$, under
perturbations of Lavaurs type
\begin{equation}\label{eq:ejn}
\varepsilon_{k,n}
=
\frac{\pi}{n}
+
\frac{\pi}{n^2}\sigma_{k,n}
+
O\Big(\frac{1}{n^{2+\alpha}}\Big),
\end{equation}
where 
$\alpha>0$ can be taken arbitrarily small (in particular, 
for simplicity, we will always assume that $\alpha<1$).
In contrast with the autonomous case, 
now
 the perturbation not only depends on $n$ but is
also
allowed to vary with
the time index $k$. 
Our main result shows that the non-autonomous system still converges to a Lavaurs
map, whose phase is obtained as an explicit weighted average of the
perturbation parameters.

Let $L_u$ denote the classical Lavaurs map associated with $f$ and phase $u\in \mathbb C$.
This map depends only on $f$ and $u$ and
has an explicit expression, see Section \ref{ss:lavaurs}.

\begin{thm}
\label{th:main}
Let
$f$, 
 $\{w_k^{(n)}\}$, and
 $\varepsilon_{k,n}$
  be as in \eqref{eq:int:f}, 
  \eqref{eq:int:wj},
  and \eqref{eq:ejn}. 
Assume that the $(\sigma_{k,n})$ as in \eqref{eq:ejn} are 
uniformly bounded.
Then
\[
w_n^{(n)}
=
L_{u_n}(w_0)
+
o(1),
\qquad n\to\infty,
\]
locally uniformly on the parabolic basin $\mathcal B_f$,
where the phase $u_n$ is given by
\[
u_n
=
\frac{1}{n}
\sum_{k=0}^{n-1}
\sigma_{k,n}\, G\!\left(\frac{k+1}{n}\right),
\qquad
G(x)=2\sin^2(\pi x).
\]
\end{thm}

In particular, whenever $u_n\to u$, the iterates converge to the Lavaurs map $L_u$.
We observe here that the function $G$ is universal: it does not depend on the higher-order terms of
$f$.

\subsection{Consequences}

We list here a few consequences
 of 
Theorem \ref{th:main}.
As a first application 
 we recover the following result, which is the core technical part of \cite{astorg2014two}.

\begin{coro}[Proposition A, \cite{astorg2014two}]\label{cor:propA}
	Take $F(z,w)=(p(z), q(w)+\frac{\pi^2}{4}z)$, where $p(z)=z-z^2+O(z^3)$ and $q(w)=w+w^2+O(w^3)$.
	Then 
	$$F^{2n+1}(p^{n^2}(z), w) = (0, L_0(w)+o(1))$$
	with local uniform convergence on $\bcal_p \times \bcal_q$.
\end{coro}

We also recover a more general version of the main result of \cite{vivas2020non}, which was 
proved in the particular case where $f(w)=\frac{w}{1-w}$, 
see also \cite{HSV26}.
Observe that, for this choice of $f$, the maps $\phi^\iota$ and $\phi^o$
as in Section \ref{ss:prelim:germs} are equal to $-1/w$
(see \eqref{eq:phi-precise} for the general expression), 
hence the map $L_0$ is equal to the identity.

\begin{coro}\label{cor:vivas}
Let $f$, $\{w_k^{(n)}\}$, and $\varepsilon_{k,n}$ be as in \eqref{eq:int:f}, \eqref{eq:int:wj}, 
	and \eqref{eq:ejn},
	 and assume that
\begin{equation}\label{eq:symm-sigma}
\sigma_{k,n}+\sigma_{n-2-k,n}=O\!\left(\frac1n\right)
\qquad\text{for every } n\in\mathbb N \text{ and } 0\le k\le n-2.
\end{equation}
Then
\[
w_n^{(n)} = L_0(w_0)+o(1).
\]
\end{coro}

We can also see the ${\sigma_{k,n}}$'s
 as random variables. For instance, we have the following probabilistic
version of Theorem \ref{th:main}.

\begin{coro}\label{cor:random}
Let $f$, $(w_{k}^{(n)})$,
 and $\varepsilon_{k,n}$ 
 be as in \eqref{eq:int:f},
 \eqref{eq:int:wj},
  and \eqref{eq:ejn}, where
 $(\sigma_{k,n})= (\sigma_k)$
  is a sequence of uniformly  bounded random variables such that
	 \[\frac{1}{n} \sum_{k=0}^{n-1} \sigma_k \to u \in \C
	 \quad \text{ almost surely}.\]
	 Then,
for every
$w_0 \in \bcal_f$,
the sequence of random variables $(w_n^{(n)})$
	converges to $L_u(w_0)$ almost surely.
\end{coro}

Finally, one can also consider the case where the $\sigma_k$'s
 arise (deterministically) from the action of a second dynamical system.

\begin{coro}\label{cor:measurableskp}
	Let $(\Omega, \mathcal{T}, \mu)$ be a probability space,  $T: \Omega \to \Omega$ be an ergodic transformation, and
	take  $\sigma \in L^\infty(\Omega, \mathcal{T}, \mu)$. Consider a	 sequence of measurable skew-products of the form 
\[f_n(z,w)=(T(z), q(w) + \eps_n(z)^2),
\quad 
\text{ where }
\eps_n(z) := \frac{\pi}{n} + \frac{\pi}{n^2} \sigma(z) + O\Big(\frac{1}{n^{2+\alpha}}\Big).\]
Then for $\mu$-almost every
$z$, 
	$$\limn \pi_2 \circ f_n^n(z,w)= L_u(w), 
	\quad \text{ 	where }
	\quad u:=\int \sigma(z) d\mu(z)$$
and  $\pi_2: \Omega \times \C$ denotes the projection on the second coordinate.
\end{coro}

Finally, 
in the spirit
of the discontinuity of the Julia sets in one-dimensional dynamics, 
we have the following result about the limits of Julia sets
(i.e., the supports of the unique measure of maximal entropy)
for endomorphisms of $\mathbb P^2 = \mathbb P^2(\mathbb C)$.

\begin{coro}	\label{cor:skewproduct}
Let $F: \ptwo \to \ptwo$ be a
holomorphic
	 endomorphism 
which, in some affine chart,
can be written in the form
	$F(z,w) = (p(z), q(w))$,
	where $p$ is a rational map and $q$ is a polynomial map of the form $q(w)=w+w^2+O(w^3)$.
	Let $F_n$ be a sequence of 
holomorphic endomorphisms of $\mathbb P^2$	
 which, in the same affine chart, have  the form
	$$F_n(z,w) = F(z,w) + \left(0,  \left( \frac{\pi}{n} + \frac{a(z)}{n^2} \right)^2\right).$$
	Assume that
	there exist three periodic points 
	$z^{(1)},z^{(2)},z^{(3)}\in J(p)$
	with periods $m_1,m_2,m_3$ such that the three numbers \[ A_j:=\frac{1}{m_j}\sum_{\ell=0}^{m_j-1}  a\!\left(p^\ell(z^{(j)})\right), \qquad j=1,2,3, \] are not collinear in $\C$.
	Then there exists 
a nonempty
	open subset $U \subset \C$ such that
	\[
	J(p) \times U \subset\ \liminf_{n\to\infty} J(F_n),
	\]
	where $J(F_n)$ denotes the Julia set of $F_n$.
\end{coro}

Corollary \ref{cor:skewproduct} will follow from a more general result
about the limit \emph{Julia-Lavaurs sets} of perturbations as in the statement, 
see Proposition \ref{p:FJ}, as well as from some thermodynamics formalism arguments, 
see Lemma \ref{lem:rotation-interior}.
Observe that no 
hyperbolicity
assumption is made on
the rational map $p$.
As a special case of Corollary \ref{cor:skewproduct}, we can consider
a rational map $p$  whose Julia set is equal to the Riemann sphere
(e.g., a Lattés map).
 In this case, 
 when $\deg p \geq 4$, we can always choose
 the function
  $a(\cdot)$
   so that the non-collinearity condition is satisfied, and
we obtain that
$\liminf_{n\to\infty} J(F_n)$ contains
 a nonempty open subset of $\mathbb P^2$.

\subsection{Structure of the paper}
In
Section~\ref{s:prelim} 
we recall
 the necessary background on Fatou coordinates and classical
parabolic implosion, as well as set some notation for the rest of the paper.
In Section \ref{s:coordinates} we construct approximate Fatou coordinates adapted to the
non-autonomous perturbations
and give the main error terms with respect to actual 
(non-autonomous) Fatou coordinates.
In Section \ref{s:proof} we prove
Theorem \ref{th:main}.
Sections \ref{s:corollaries}
and \ref{s:cor-new}
are devoted to the proofs of the corollaries of the main theorem.

\subsection*{Acknowledgements}

This work was partially done during the authors' participation in the thematic semester
\emph{Holomorphic Dynamics and Geometry of Surfaces} 
which took place in Toulouse in Spring 2025. We would like to thank the organizers
and the Institut de Mathématiques de
Toulouse for the warm welcome and the excellent work conditions.

This project has received funding from
 the French government through the Programme
 Investissement d'Avenir
(ANR QuaSiDy /ANR-21-CE40-0016,
ANR PADAWAN /ANR-21-CE40-0012-01,
ANR DynAtrois / ANR-24-CE40-1163,
ANR TIGerS/ANR-24-CE40-3604)
and from the MIUR Excellence Department Project awarded to the Department of Mathematics, University of Pisa, CUP I57G22000700001.
The second
author is affiliated to the GNSAGA group of INdAM.
Both authors were part of the PHC Galileo project G24-123.

\section{Preliminaries and notations}\label{s:prelim}

In this section we recall the necessary background on parabolic germs,
Fatou coordinates, and classical (autonomous) parabolic implosion.
We also fix notation that will be used throughout the paper.

\subsection{Parabolic germs and Fatou coordinates}\label{ss:prelim:germs}

Let
\[
f(w)=w+w^2+aw^3 + O(w^4)
\]
be a holomorphic germ at the origin.
The fixed point at $w=0$ is parabolic with multiplier $1$ and one attracting
and one repelling direction.  For a sufficiently small $r>0$, define the 
\emph{attracting} and \emph{repelling petals}
 as
\[
P^\iota = \{|w+r|<r\} \quad \text{ and } \quad
P^o  = \{|w-r|<r\}  \]
The attracting petal $P^\iota$ is forward invariant and we have $f^n (w)\to 0$
for every $w\in P^\iota$. The 
\emph{parabolic basin}
 $\mathcal B_f$
consists of all points $w$ with $f^{n}(w)\to 0$ and is equal to 
$\cup_n f^{-n} (P^\iota)$.

It is classical that there exist attracting and repelling Fatou coordinates
on the petals,
that is, holomorphic maps
$\phi^{\iota} \colon \mathcal P^\iota\to \mathbb C$ and
 $\phi^{o}\colon P^o\to \mathbb C$
such that
\begin{equation}\label{eq:fatou-eq}
\phi^{\iota}(f(w))=\phi^{\iota}(w)+1,
\qquad
\phi^{o}(f(w))=\phi^{o}(w)+1.
\end{equation}
The image of $P^\iota$  by $\phi^\iota$ (resp.\ the image of 
$P^o$  by $\phi^o$)
contains a right 
(resp.\ left) half plane.
Moreover, 
$\phi^\iota$ and $\phi^o$ can be chosen\footnote{The maps $\phi^\iota$ and $\phi^o$
are unique up to an additive constant, which we fix 
by requiring the asymptotics
\eqref{eq:phi-precise}, where $\log(\cdot)$ 
denotes the principal branch of the logarithm.}
 to satisfy 
\begin{equation}
\label{eq:phi-precise}
\begin{cases}
\phi^\iota(w)=-\frac{1}{w} +(1-a) \log (- w) + o(1)
\\
\phi^o(w)=-\frac{1}{w} 
+
(1-a) \log w + o(1),
\end{cases}
\quad w\to 0.
\end{equation}
Thanks to \eqref{eq:fatou-eq}, $\phi^\iota$ can 
also  be extended to $\mathcal B_f$. Similarly,
the inverse $(\phi^o)^{-1}$ can be extended to the complex plane.

\subsection{Lavaurs maps}\label{ss:lavaurs}

The now
classical Douady--Lavaurs construction
\cite{lavaurs1989systemes,Douady94Julia}
 associates to $f$
a one-parameter family of holomorphic maps $(L_u)_{u\in\C}$
on $\mathcal B_f$
defined by
\[
L_u
=
(\phi^{o})^{-1}
\circ
T_u
\circ
\phi^{\iota},
\]
where $T_u(\zeta)=\zeta+u$ denotes translation by $u$.
These maps arise as limits of suitable perturbations of $f$.
More precisely, if one considers autonomous perturbations
\[
f_{\varepsilon}(w)=f(w)+\varepsilon^2
\]
and a sequence $\{\varepsilon_n\}$ with
$\varepsilon_n=\pi/n + \pi u/n^2 + o(1/n^2)$,
  then
\[
f_{\varepsilon_n}^{\circ n}
\longrightarrow
L_u
\]
locally uniformly on $\mathcal B_f$.

\section{Approximate Fatou coordinates for non-autonomous perturbations}
\label{s:coordinates}

In this section we construct approximate Fatou coordinates adapted to the
non-autonomous 
iteration
\eqref{eq:int:wj}, where
the  $\varepsilon_{k,n}$
are as in \eqref{eq:ejn}.
The goal is to conjugate the dynamics
of $f_{k,n}:= f + \varepsilon_{k,n}^2$
 to an approximate translation, with an error that can be explicitly
controlled with respect to $k$ and $n$.
 In order to simplify some notations, we will
  drop the index $n$ whenever possible and denote $f_{k}=f_{k,n}$ and $w_{k}=w_k^{(n)}$.
Throughout this section,
 we fix a sequence of integers $(k_n)_{n\ge1}$ such that
$1\le k_n\le n/2$, $k_n\to\infty$, and $k_n=o(n)$.
The precise choice of $k_n$ 
needed in the proof of Theorem \ref{th:main}
will be made in Section \ref{s:proof}.

\subsection{The coordinate $\psi$}
We begin with the following
 lemma, which
  gives an approximation of the fixed points of the maps $f_k$.
  Recall that we assume $0<\alpha<1$ in \eqref{eq:ejn}.

\begin{lem}\label{l:zeta}
Set $\zeta^\pm:= \pm \frac{i\pi}{n} + \frac{a\pi^2}{2n^2}$.
	Then 
	\[f_k(\zeta^\pm) = \zeta^\pm + \frac{\delta_k}{n^3}+ O\Big(
	\frac{1}{n^{3+\alpha}}
	\Big),
\qquad \text{ where } \qquad \delta_k:=2\pi^2 \sigma_k.\]
\end{lem}

\begin{proof}
Let us set
 $\zeta_k^{\pm} := \pm \frac{i\pi}{n} + \frac{\alpha_2^\pm(k)}{n^2}$.
  A direct computation gives
\[\begin{aligned}
f_k (\zeta_k^\pm) - \zeta_{k+1}^{\pm}
 = &
\zeta_k^\pm + ( \zeta_k^\pm)^2 + a (\zeta_k^\pm)^3 + O( (\zeta_k^\pm)^4)  + \eps_{k,n}^2 - \zeta_{k+1}^{\pm}\\
 = & \pm \frac{i\pi}{n} + \frac{\alpha_2^\pm(k)}{n^2}- \frac{\pi^2}{n^2}
\pm \frac{2i\pi \alpha_2^\pm (k)}{n^3} \mp a \frac{i\pi^3}{n^3}		\\
& + \frac{\pi^2}{n^2} + \frac{2\pi^2\sigma_k}{n^3}
 \mp \frac{i\pi}{n} - \frac{\alpha_2^\pm (k+1)}{n^2}
  + O\Big(\frac{1}{n^{3+\alpha}}
  \Big)\\
 =  &  \frac{\alpha_2^\pm (k) - \alpha_2^\pm (k+1)}{n^2}  
+
  \frac{2\pi^2\sigma_k \pm 2i\pi \alpha_2^\pm (k) \mp ai\pi^3}{n^3}
  + O\Big(
  \frac{1}{n^{3+\alpha}}
  \Big).
\end{aligned}\]
Choosing $\alpha_2^\pm (k) = \frac{a\pi^2}{2}$ for all $k\geq 0$, 
we obtain $\zeta^{\pm}_k = \zeta^\pm$ and 
\[
f_k (\zeta_k^\pm) - \zeta_{k+1}^{\pm}
=
\frac{2\pi^2 \sigma_k}{n^3} + O\Big(\frac{1}{n^{3+\alpha}}
\Big),
\]
which is the desired expression.
\end{proof}

Observe that, in particular, $\zeta^\pm$ as above do not depend on $k$
(we allowed for a dependence $\alpha_2(k)$ in the proof to show that
adding such a dependence would not have improved the estimates).

\begin{defi}
We denote
 \[\psi(w):=\frac{1}{2i\pi} \log \Big( \frac{w-\zeta^+}{w-\zeta^-}\Big)
 \quad \text{ and }
 \quad
	A_k := A_k(w):=\psi(w_{k+1}) - \psi(w_k).\]
\end{defi}
{Here and in the rest of the paper, we denote by
 $\log (\cdot)$ 
  the principal branch of the logarithm.}
  Note that the map $\psi$ sends the real axis to the segment $(0,1)$.
 Circles through $\zeta^+$ and $\zeta^-$ are mapped to vertical lines in the strip above.

The map $\psi$ is a first approximation of the non-autonomous Fatou coordinates
that we will use. We will now
 give an estimate for 
the error term $A_k$, for $w$ in a suitable domain. 
We begin 
by giving the following expression for the inverse of $\psi$.

\begin{lem}\label{lem:cot}
	We have
	$$\psi^{-1}(W) = -\frac{\pi}{n} \cot(\pi W) + O\Big(\frac{1}{n^2}\Big).$$
\end{lem}

\begin{proof}
Recall that $\cot z = i \frac{e^{iz} + e^{-iz}}{e^{iz}- e^{-iz}}$. Moreover,
by the expression of $\psi$ we see that
\[
\psi^{-1} (W) = \frac{e^{i\pi W} \zeta^- - e^{-i\pi W} \zeta^+ }{e^{i\pi W} - e^{-i\pi W}}.
\]
It follows from the definition of
$\zeta^\pm$
that
\[
\psi^{-1} (W)  = -\frac{\pi}{n} \cot (\pi W) + \frac{a\pi^2}{2n^2},
\]
which gives the desired estimate.
\end{proof}

\begin{defi}
	We define the rectangle 
	$$R_n:=\{W \in \C : \re(W) \in \Big(\frac{k_n}{10n}, 1 - \frac{k_n}{10n}\Big), \quad 
	\im(W) \in (-1,1) \}.$$
\end{defi}

Observe that $R_n$ is contained in the image of $\psi$ for every $n$.
For a fixed $n$ we will mainly consider orbits $(w_k^{(n)})$ such that
$\psi(w_k^{(n)})\in R_n$ for the relevant times $k$.
The following lemma gives a simple estimate that will 
be able to apply to such orbits.

\begin{lem}\label{l:rectangle}
There exists a constant $c>0$
 such that
$|\psi^{-1}(W)-\zeta^\pm| \geq \frac{c}{n}$
for every
$n$ and
 $W \in R_n$.
\end{lem}

\begin{proof}
By Lemma \ref{l:zeta},
it is enough to prove that 
\begin{center}
$\displaystyle |\psi^{-1} (W) \pm \frac{i\pi}{n}| \geq \frac{c}{n}$
for all $W$ with $|\im (W)| \leq 1$,
\end{center}
for some suitable constant $c$ as in the statement.
By Lemma \ref{lem:cot},
 this reduces to 
\begin{center}
$ |\cot(\pi W) \pm i| \geq c$ for some positive $c$.
\end{center}
The above holds since $\pm i$ are not in the image of the map
$W \mapsto \cot (\pi W)$, hence $n\zeta^{\pm}$
are 
uniformly far from 
 the image of
$\overline{R_n}$ under $n\psi^{-1}$.
\end{proof}

\begin{prop}\label{prop_Ak}
For every $n\in \mathbb N$, 
$0\leq k < n$, and
$w \in \mathcal B_f$ with $\psi(w_{k}), \psi(w_{k+1})\in R_n$
 	we have
	$$A_k (w)=\frac{1}{n}
	- (1-a) \frac{w_k}{n}+ H_{k,n}(w_{k+1})+O\Big(\frac{w_k^2}{n},
	\frac{1}{n^{2+\alpha}}
\Big)$$
	where
	\begin{equation}\label{eq:Hkn}
	H_{k,n}(w):=\frac{\delta_k}{2i\pi n^3} \left(\frac{1}{w-\zeta_+}- \frac{1}{w-\zeta^-}\right)
	\end{equation}
	and the estimate is locally uniform
	in $w \in \mathcal B_f$. 
\end{prop}

\begin{proof}
By definition, we have
\[
A_k
=
\frac{1}{2i\pi}\log\frac{w_{k+1}-\zeta^+}{w_{k+1}-\zeta^-}
-
\frac{1}{2i\pi}\log\frac{w_k-\zeta^+}{w_k-\zeta^-}.
\]
Define
\[
A'_k:=
\frac{1}{2i\pi}\log\frac{w_{k+1}-f_k(\zeta^+)}{w_{k+1}-f_k(\zeta^-)}
-
\frac{1}{2i\pi}\log\frac{w_k-\zeta^+}{w_k-\zeta^-}.
\]
We split the proof into two main 
steps, where we estimate $A'_k$ 
(the main term)
and $A_k-A'_k$
(the correction term), respectively.

\medskip

\noindent
\textbf{Step 1. Estimate of the main term $A'_k$.}
We claim that we have
\begin{equation}\label{eq:estimate_A'k}
A'_k
=
\frac{1}{n}
-
(1-a)\frac{w_k}{n}
+
O\Big(\frac{w_k^2}{n},\frac{1}{n^3}\Big).
\end{equation}
for every $n$ sufficiently large and every $0\leq k < n$.
 
\begin{proof}[Proof of Step 1]
Since $w_{k+1}=f_k(w_k)$, we can rewrite $A'_k$ 
as
\[
A'_k
=
\frac{1}{2i\pi}\log\frac{f_k(w_k)-f_k(\zeta^+)}{f_k(w_k)-f_k(\zeta^-)}
-
\frac{1}{2i\pi}\log\frac{w_k-\zeta^+}{w_k-\zeta^-}.
\]
As $f_k$ differs from $f$ by an additive constant, we deduce
\[
A'_k
=
\frac{1}{2i\pi}\log\frac{f(w_k)-f(\zeta^+)}{f(w_k)-f(\zeta^-)}
-
\frac{1}{2i\pi}\log\frac{w_k-\zeta^+}{w_k-\zeta^-}.
\]
Equivalently, we have
\[
A'_k
=
\frac{1}{2i\pi}\log\frac{f(w_k)-f(\zeta^+)}{w_k-\zeta^+}
-
\frac{1}{2i\pi}\log\frac{f(w_k)-f(\zeta^-)}{w_k-\zeta^-}.
\]

Set
\[
\xi^\pm(t):=\pm i\pi t+\frac{a\pi^2}{2}t^2,
\qquad t:=\frac{1}{n},
\]
so that $\zeta^\pm=\xi^\pm(1/n)$. Consider the function
\[
Q(x,\xi):=\frac{f(x)-f(\xi)}{x-\xi}.
\]
Since $f$ is holomorphic, $Q$ extends holomorphically across the diagonal $x=\xi$, and
\[
Q(0,0)=f'(0)=1.
\]
Therefore, on a sufficiently small neighborhood of $(0,0)$, 
we may choose a holomorphic
branch of $\log Q(x,\xi)$. 
We then define
\[
\mathcal A'(x,t)
:=
\frac{1}{2i\pi}\Big(\log Q(x,\xi^+(t))-\log Q(x,\xi^-(t))\Big).
\]
This is a holomorphic function of $(x,t)$ in a neighborhood of $(0,0)$, independent of
$k$ and $n$, and by construction we have
\[
A'_k=\mathcal A'(w_k,1/n).
\]

We now expand $\mathcal A'(x,t)$ at $t=0$. First observe that
$\xi^\pm(-t)=\xi^{\mp}(t)$,
hence
$\mathcal A'(x,-t)=-\mathcal A'(x,t)$.
So $\mathcal A'(x,t)$ is odd in $t$, and in particular all even Taylor coefficients in $t$
vanish.
Next, 
differentiating with respect to $t$, we get
\[
\partial_t \log Q(x,\xi^\pm(t))
=
-(\xi^\pm)'(t)
\left(
\frac{f'(\xi^\pm(t))}{f(x)-f(\xi^\pm(t))}
-
\frac{1}{x-\xi^\pm(t)}
\right).
\]
Since $\xi^\pm(0)=0$, $(\xi^\pm)'(0)=\pm i\pi$, and $f'(0)=1$, we obtain
\[
\partial_t \log Q(x,\xi^\pm(t))\Big|_{t=0}
=
\pm i\pi\left(\frac{1}{x}-\frac{1}{f(x)}\right).
\]
Therefore, we obtain
\[
\partial_t\mathcal A'(x,0)
=
\frac{1}{x}-\frac{1}{f(x)}.
\]
We conclude that 
the Taylor expansion 
of $\mathcal A'$
at $t=0$ has the
form
\begin{equation}\label{eq:development}
\mathcal A'(x,t)
=
\left(\frac{1}{x}-\frac{1}{f(x)}\right)t
+O(t^3),
\end{equation}
where in particular the error is uniform
 for $x$ in a sufficiently small fixed neighbourhood of $0$.

\medskip

We now
verify
 that $x=w_k$ belongs to such a neighbourhood, 
 for every $n$ sufficiently large. 
 Indeed, by the assumption
$\psi(w_k)\in R_n$ and the formula for $\psi^{-1}$, we have
\[
w_k=\psi^{-1}(\psi(w_k))
=
-\frac{\pi}{n}\cot(\pi\psi(w_k))+O\Big(\frac{1}{n^2}\Big).
\]
Since $\psi(w_k)\in R_n$, the distance of $\psi(w_k)$ to $\mathbb Z$ is bounded below by
$k_n/(10n)$, hence
\[
\left|\cot(\pi\psi(w_k))\right| \le C\frac{n}{k_n}
\]
for some constant $C>0$ independent of $k$ and $n$. Thus
\[
w_k=O\Big(\frac{1}{k_n}\Big)+O\Big(\frac{1}{n^2}\Big).
\]
Since $k_n\to\infty$, it follows that for $n$ large enough all such $w_k$ belong to a fixed
small neighbourhood of $0$. 

\medskip
Thanks to the above,
 we can 
 evaluate \eqref{eq:development} 
 at
$(x,t)=(w_k,1/n)$ and obtain
\begin{equation}\label{eq:A'k-pre}
A'_k
=
\left(\frac{1}{w_k}-\frac{1}{f(w_k)}\right)\frac{1}{n}
+
O\Big(\frac{1}{n^3}\Big).
\end{equation}

Using the expansion
$f(w)=w+w^2+aw^3+O(w^4)$,
we get
\[
\frac{1}{f(w)}
=
\frac{1}{w}\,\frac{1}{1+w+aw^2+O(w^3)}
=
\frac{1}{w}\Big(1-w+(1-a)w^2+O(w^3)\Big),
\]
hence
\[
\frac{1}{w}-\frac{1}{f(w)}
=
1-(1-a)w+O(w^2).
\]
Applying this at $w=w_k$ and inserting it into
 \eqref{eq:A'k-pre}, 
we obtain
\eqref{eq:estimate_A'k}.
\end{proof}

\noindent
\textbf{Step 2. Estimate of the correction  term $A_k-A'_k$.}
We claim that we have
\begin{equation}\label{eq:diff}
A_k-A'_k=H_{k,n}(w_{k+1})+O\Big(\frac{1}{n^{2+\alpha}}\Big)
\end{equation}
for every $n$ sufficiently large and every $0\leq k < n$.

\begin{proof}[Proof of Step 2]
By definition, we have
\[
A_k-A'_k
=
\frac{1}{2i\pi}\log\frac{w_{k+1}-\zeta^+}{w_{k+1}-\zeta^-}
-
\frac{1}{2i\pi}\log\frac{w_{k+1}-f_k(\zeta^+)}{w_{k+1}-f_k(\zeta^-)},
\]
which we can rewrite as
\begin{equation}\label{eq:diff-pre}
A_k-A'_k
=
\frac{1}{2i\pi}
\left(
\log\frac{w_{k+1}-\zeta^+}{w_{k+1}-f_k(\zeta^+)}
-
\log\frac{w_{k+1}-\zeta^-}{w_{k+1}-f_k(\zeta^-)}
\right).
\end{equation}

Set
\[
\eta^\pm_{k,n}:=f_k(\zeta^\pm)-\zeta^\pm.
\]
By Lemma
\ref{l:zeta},
 we have
\[
\eta^\pm_{k,n}
=
\frac{\delta_k}{n^3}+O\Big(\frac{1}{n^{3+\alpha}}\Big),
\]
where the error term is uniform
 in $k$. Since the sequence $(\delta_k)$ is uniformly bounded, this implies
\begin{equation}\label{eq:2:eta}
\eta^\pm_{k,n}=O\Big(\frac{1}{n^3}\Big),
\end{equation}
where again the error term is
uniform in $k$.
On the other hand, since $\psi(w_{k+1})\in R_n$, Lemma
\ref{l:rectangle} gives
\begin{equation}\label{eq:2:rect}
|w_{k+1}-\zeta^\pm|\ge \frac{c}{n}
\end{equation}
for some $c>0$ independent of $k$ and $n$. Hence
\[
\left|\frac{\eta^\pm_{k,n}}{w_{k+1}-\zeta^\pm}\right|
\le C\frac{1/n^3}{1/n}
=
O\Big(\frac{1}{n^2}\Big)
\]
uniformly in $k$ and $n$. In particular, for $n$ large enough
we obtain
\[
\log\frac{w_{k+1}-\zeta^\pm}{w_{k+1}-f_k(\zeta^\pm)}
=
-\log\left(1-\frac{\eta^\pm_{k,n}}{w_{k+1}-\zeta^\pm}\right)
=
\frac{\eta^\pm_{k,n}}{w_{k+1}-\zeta^\pm}
+
O\!\left(\frac{|\eta^\pm_{k,n}|^2}{|w_{k+1}-\zeta^\pm|^2}\right).
\]
Using again
\eqref{eq:2:eta} and \eqref{eq:2:rect},
 the
error term above satisfies
\[
O\!\left(\frac{|\eta^\pm_{k,n}|^2}{|w_{k+1}-\zeta^\pm|^2}\right) 
=
O\Big(\frac{1/n^6}{1/n^2}\Big)=O\Big(\frac{1}{n^4}\Big).
\]
Moreover, we also have
\[
\frac{\eta^\pm_{k,n}}{w_{k+1}-\zeta^\pm}
=
\frac{\delta_k}{n^3}\frac{1}{w_{k+1}-\zeta^\pm}
+
O\!\left(\frac{1}{n^{3+\alpha}}\cdot \frac{1}{|w_{k+1}-\zeta^\pm|}\right)
=
\frac{\delta_k}{n^3}\frac{1}{w_{k+1}-\zeta^\pm}
+
O\!\left(\frac{1}{n^{2+\alpha}}\right).
\]
Since $O(n^{-4})=O(n^{-2-\alpha})$ (recall that  we always assume
$0<\alpha<1$), we conclude that
\begin{equation}\label{eq:diff-part}
\log\frac{w_{k+1}-\zeta^\pm}{w_{k+1}-f_k(\zeta^\pm)}
=
\frac{\delta_k}{n^3}\frac{1}{w_{k+1}-\zeta^\pm}
+
O\!\left(\frac{1}{n^{2+\alpha}}\right),
\end{equation}
uniformly in $k$ and $n$.

\medskip

Combining \eqref{eq:diff-pre}  and \eqref{eq:diff-part}, we obtain
\[
A_k-A'_k
=
\frac{\delta_k}{2i\pi n^3}
\left(
\frac{1}{w_{k+1}-\zeta^+}
-
\frac{1}{w_{k+1}-\zeta^-}
\right)
+
O\Big(\frac{1}{n^{2+\alpha}}\Big).
\]
By the definition of $H_{k,n}$, this
gives \eqref{eq:diff} and concludes the proof.
\end{proof}

Combining
the estimates in Steps 1 and 2,
we obtain
\[
A_k
=
\frac{1}{n}
-
(1-a)\frac{w_k}{n}
+
H_{k,n}(w_{k+1})
+
O
\left(\frac{w_k^2}{n},\frac{1}{n^3},\frac{1}{n^{2+\alpha}}\right).
\]
Since $1/n^3=O(1/n^{2+\alpha})$, this simplifies to
\[
A_k
=
\frac{1}{n}
-
(1-a)\frac{w_k}{n}
+
H_{k,n}(w_{k+1})
+
O 
\left(\frac{w_k^2}{n},\frac{1}{n^{2+\alpha}}\right),
\]
which is the desired estimate.
Observe that
all the implicit constants above are independent of $k$ and $n$, and are uniform for
$w$ 
 in a compact subset of $\mathcal B_f$. This proves the local uniformity in
$w$.
\end{proof}

\subsection{The coordinate $\phi$}
We will now 
post-compose
 $\psi$ with a further suitable change of coordinate
to
get rid of the error term in $(1-a)w_k/n$ appearing in 
Proposition \ref{prop_Ak}. Observe that this step is not necessary when $a=1$, as $\phi=\psi$ in this case.

\begin{defi}
We set
\[\chi(W):=W - \frac{1}{n} (1-a) \log \sin(\pi W),
\quad
 \phi := \chi \circ \psi,
 \quad  \text{and}
 \quad
\tilde A_k:= \phi (w_{k+1}) - \phi (w_k).\]
\end{defi}

For convenience,
we will denote
$U_k := \psi (w_k)$
 and  $W_k := \phi (w_k)$.

\begin{lem}\label{l:psi-phi}
The following properties hold.
\begin{enumerate}
\item 
For every $n$ and every $W\in R_n$, the inverse branch of $\phi$
defined by $\phi^{-1}(W):=\psi^{-1}(\chi^{-1}(W))$ satisfies
\[
\phi^{-1}(W)= -\frac{\pi}{n}\cot(\pi W)+O\!\left(\frac1{n^2}\right)
+O\!\left(\frac{\log(n/k_n)}{k_n^2}\right).
\]
\item 
There exists a constant $c>0$ such that
\[
|\phi^{-1}(W)-\zeta_\pm|\ge \frac{c}{n}
\qquad\text{for every $n$ and every $W\in R_n$.}
\]
\end{enumerate}
\end{lem}

\begin{proof}
(1)
Fix $W\in R_n$ and set $U:=\chi^{-1}(W)$, so that $\phi^{-1}(W)=\psi^{-1}(U)$.
Since $\re (W) \in\bigl(\frac{k_n}{10n},1-\frac{k_n}{10n}\bigr)$ and $|\im( W)|<1$,
 we have
$|\sin(\pi W)|\asymp {\rm dist}(W,\mathbb Z)\gtrsim \frac{k_n}{n}$, hence
\[|\log\sin(\pi W)|\lesssim 1+\log(n/k_n),\]
 where the implicit constants
are independent of $n$ and $W$.
Therefore, from $W=U-\frac{1-a}{n}\log\sin(\pi U)$ we obtain
\[
U=W+O\!\left(\frac{1+\log(n/k_n)}{n}\right).
\]

By Lemma
\ref{lem:cot},
we have
\[
\phi^{-1}(W)=\psi^{-1}(U)
=-\frac{\pi}{n}\cot(\pi W)+O\!\left(\frac{1}{n^2}\right)
+O\!\left(\frac{1+\log(n/k_n)}{k_n^2}\right),
\]
where we used that $|\cot'(\pi Z)|=O\bigl((n/k_n)^2\bigr)$ on $R_n$ and the previous bound on $|U-W|$.

\medskip

(2)
The bound follows
using the development
 $\zeta_\pm=\pm i\pi/n+O(1/n^2)$
 (see Lemma \ref{l:zeta})
  and the lower bound
for
$|\cot(\pi W)\pm i|$
given by Lemma \ref{l:rectangle}.
\end{proof}

\begin{prop}\label{prop:estimate-central}
For every $n\in \mathbb N$, 
$0\leq k < n$, and
$w \in \mathcal B_f$ with $W_k, W_{k+1}\in R_n$
	we have
	$$\tilde A_k = \frac{1}{n}+  H_{k,n}(-\frac{\pi}{n}\cot(\pi W_{k+1})) 
	+
	O\Big(\frac{w^2_k}{n},
	 \frac{1}{n^{2+\alpha}},
	 \frac{\log(n/k_n)}{n\,k_n^2}\Big).$$
\end{prop}

\begin{proof}
 By
 the definitions of
 $\phi$, 
 $\tilde A_k$, and $A_k$,  
 we see that
 \begin{equation}\label{eq:forAktilde}
 \begin{aligned}
\tilde A_k & = 
W_{k+1} - W_k=
A_k - \frac{1}{n}(1-a)
\big(
\log \sin (\pi U_{k+1}) - \log \sin (\pi U_k)
\big) \\
&=
\frac{1}{n}- (1-a) \frac{w_k}{n}+ H_{k,n}(w_{k+1})
- \frac{1}{n}(1-a)
\log \frac{ \sin (\pi U_{k+1})}{ \sin (\pi U_k)}
+O\Big(\frac{w_k^2}{n},
 \frac{1}{n^{2+\alpha}}\Big).
\end{aligned} 
 \end{equation}
Hence, we need to
compare 
$H_{k,n}(w_{k+1})$ with 
$ H_{k,n}(-\frac{\pi}{n}\cot(\pi W_{k+1}))$ and
 estimate the term
 $\log \frac{\sin (\pi U_{k+1})}{\sin (\pi U_k)}$.
 For convenience, we will set
 $\widehat w_k:=-\frac{\pi}{n}\cot(\pi W_k)$.

\medskip

 \noindent {\bf Claim 1.}
We have
\[
H_{k,n}(w_{k+1})=H_{k,n}(\widehat w_{ k+1})+O\!\left(\frac1{n^3},
\frac{\log(n/k_n)}{n\,k_n^2}\right).
\]

 \begin{proof}
By the definition
\eqref{eq:Hkn} of $H_{k,n}(w)$, we have
\[
H'_{k,n}(w)=
\frac{\delta_k}{2i\pi n^3}\left(-\frac{1}{(w-\zeta_+)^2}+\frac{1}{(w-\zeta_-)^2}\right).
\]
Since
$W_{k+1}\in R_n$, by Lemma
\ref{l:psi-phi}(2)
 we have $|w_{k+1}-\zeta_\pm|\ge c/n$.
Moreover, $\widehat w_{k+1}=\phi^{-1}(W_{k+1})+O(\frac{1}{n^2},\frac{\log(n/k_n)}{k_n^2})$ by Lemma

\ref{l:psi-phi}(1),
hence also
$|\widehat w_{k+1}-\zeta_\pm|\ge c'/n$ for some $c'>0$
(using Lemmas
\ref{l:rectangle} and \ref{l:psi-phi}(2)).
Therefore, for all $w$ on the segment joining $w_{k+1}$ to $\widehat w_{k+1}$ we have
$|w-\zeta_\pm|\ge c''/n$
for some $c''>0$, 
and thus
\[
|H'_{k,n}(w)|\le \frac{C}{n^3}\left(\frac{n^2}{(c'')^2}+\frac{n^2}{(c'')^2}\right)\le \frac{C_1}{n},
\]
for some positive constants $C,C_1$, where we used 
the fact that the sequence
 $(\delta_k)$ is uniformly bounded.
By the mean value theorem, we have
\[
|H_{k,n}(w_{k+1})-H_{k,n}(\widehat w_{k+1})|
\le \sup_{[w_{k+1},\widehat w_{k+1}]}|H'_{k,n}|\cdot |w_{k+1}-\widehat w_{k+1}|
\le \frac{C_1}{n}\cdot O\!\left(\frac{1}{n^2}, \frac{\log(n/k_n)}{k_n^2}\right),
\]
which gives the assertion.
\end{proof}

\noindent {\bf Claim 2.}
We have $ \log \sin (\pi U_{k+1}) - \log \sin (\pi U_k) = - w_k + O({w^2_k}, \frac{1}{n^2})$.

\begin{proof}
We have 
\begin{equation}\label{eq:cl2-0}
\begin{aligned}
 \log \sin (\pi U_{k+1}) - \log \sin (\pi U_k)
 & = 
 \log \frac{\sin (\pi ( U_{k}+A_k))}{\sin (\pi U_k)}\\
 &= 
 \log \frac{\cos (\pi A_k)\sin (\pi U_k) + \sin (\pi A_k) \cos (\pi U_k)}{\sin (\pi U_k)}\\
 & = 
  \log (\cos (\pi A_k) + \sin (\pi A_k) \cot (\pi U_k)).
 \end{aligned}
\end{equation}
 Since $A_k=O(1/n)$
(see Proposition \ref{prop_Ak}),
 we can expand
\[
\cos(\pi A_k)=1+O(A_k^2)\quad
\text{and}\quad
\sin(\pi A_k)=\pi A_k+O(A_k^3),
\]
which give
\begin{equation}\label{eq:cl2-1}
\cos(\pi A_k)+\sin(\pi A_k)\cot(\pi U_k)
=1+\pi A_k\cot(\pi U_k)+O\!\bigl(A_k^2(1+\cot^2(\pi U_k))\bigr).
\end{equation}
Recalling that $U_k = \psi (w_k)$,
we deduce 
from 
Lemma \ref{lem:cot}
that
\[
w_k=\psi^{-1}(U_k)=-\frac{\pi}{n}\cot(\pi U_k)+O\!\left(\frac1{n^2}\right),
\quad\text{hence}\quad
\cot(\pi U_k)=-\frac{n}{\pi}w_k+O\!\left(\frac1n\right).
\]
Combining this with
expansion of $A_k$ given in Proposition
\ref{prop_Ak} gives
\begin{equation}\label{eq:cl2-2}
\begin{aligned}
\pi A_k\cot(\pi U_k)
& =\pi\Bigl(\frac1n-\frac{(1-a)w_k}{n}
+O\Bigl(\frac{w_k^2}{n}
 \frac1{n^3}\Bigr)\Bigr)
\Bigl(-\frac{n}{\pi}w_k+O\!\Bigl(\frac1n\Bigr)\Bigr)\\
& =-w_k+(1-a)w_k^2+O(w_k^2)+O\!\left(\frac1{n^2}\right).
\end{aligned}
\end{equation}
Moreover, we also have
\begin{equation}\label{eq:cl2-3}
A_k^2(1+\cot^2(\pi U_k))
=O\!\left(\frac1{n^2}\right)+O\!\left(\frac{\cot^2(\pi U_k)}{n^2}\right)
=O\!\left(\frac1{n^2}\right)+O(w_k^2),
\end{equation}
using again
the expression $\cot(\pi U_k)=-(n/\pi)w_k+O(1/n)$.
Therefore, we deduce from
\eqref{eq:cl2-1},
 \eqref{eq:cl2-2}, and \eqref{eq:cl2-3} that
\[
\cos(\pi A_k)+\sin(\pi A_k)\cot(\pi U_k)
=1-w_k+(1-a)w_k^2+O(w_k^2)+O\!\left(\frac1{n^2}\right).
\]
We conclude 
from \eqref{eq:cl2-0} and the above expression that
\[
\log\sin(\pi U_{k+1})-\log\sin(\pi U_k)
=-w_k+O(w_k^2)+O\!\left(\frac1{n^2}\right),
\]
as claimed
\end{proof}

The assertion follows from \eqref{eq:forAktilde} and 
the above claims.
\end{proof}

\section{Proof of Theorem \ref{th:main}}\label{s:proof}

In this section we prove 
our main Theorem \ref{th:main}. The proof 
follows the orbits and is divided into three main steps. In Section 
\ref{ss:entering} we consider the first  $\sim k_n$ points of the orbits,
which still look close to
the autonomous system $f$.
In Section \ref{ss:passing} we study the central part of the orbit,
from $k\sim k_n$ to $k\sim n-k_n$, where the perturbations $\varepsilon_{k,n}$ dominate. This is 
where the estimates of the previous sections will be mainly used 
and 
the accumulated errors
 leading to the formula for the phase in the statement will appear.
 Finally, 
 in Section \ref{ss:exiting}
 we consider the last $\sim k_n$ points of the orbit, for which again the main contribution is given by the 
 dynamics of $f$.
 In Section \ref{ss:end} we will put all the estimates together, completing the proof of Theorem \ref{th:main}.

\medskip

We fix below a
 point $w_0\in \mathcal B_f$. All the estimates will be uniform
for $w$ in a given compact neighbourhood of $w_0$ in $\mathcal B_f$.
From now on, we will assume that the sequence $k_n$ satisfies
$n = o(k_n^2/{ \log (n/k_n)})$ 
and
$k_n =o(n^{2/3})$, i.e., that
\begin{equation}\label{eq:cond-kn}
\frac{k_n^3}{n^2},\, \frac{n
{\log (n/k_n)}}{k_n^2}\to 0, \qquad n\to \infty.
\end{equation}
For instance, we can
fix $1/2 < \beta < 2/3$ and 
take $k_n := \lfloor n^{\beta}\rfloor$.
Observe that, with these assumptions, the error term in 
Proposition \ref{prop:estimate-central} becomes $o(1/n^2)$.

\subsection{Entering the eggbeater}\label{ss:entering}
In this section we compare the readings
 of the first approximately $k_n$ points of an orbit
in the
approximate Fatou coordinate $\phi$ introduced above and
the actual Fatou
 coordinate
  $\phi^\iota$ (see Section \ref{s:prelim}). We start with the following lemma about
  $\phi^\iota$.

\begin{lem}\label{l:phi-iota-kn}
	We have
$w_{k_n} =- \frac{1+o(1)}{k_n}$
and	
	 $\phi^\iota(w_{k_n})=\phi^\iota(w_0)+k_n + o(1)$. 
\end{lem}

\begin{proof}
Recall that $w_{k+1} = f(w_k) + \eps^{2}_{k,n}$, and that $|\eps_{k,n}^2| = O(\frac{1}{n^2})$.
Moreover, we have $\phi^\iota(f(w)) = \phi^\iota (w)+1$ and $(\phi^\iota)'(w_k)=O(\frac{1}{w_k^2})$.
These identities and estimates give
\[
\begin{aligned}
\phi^\iota (w_{k+1}) &  =
 \phi^\iota ( f_k(w_k)) =
\phi^\iota( f(w_k) + \eps^2_{k,n})\\
& =
\phi^\iota (f(w_k))  + O( | \eps^2_{k,n}| |(\phi^\iota)' (f(w_k))|)\\
&= \phi^\iota (w_k) + 1 + O\Big(\frac{1}{ |w_{k+1}|^2 \cdot n^2}\Big).
\end{aligned}
\]
By induction and the asymptotic
\eqref{eq:phi-precise} of $\phi^\iota$, 
the above shows that $w_{k} =-\frac{1+o(1)}{k}$
for every $0\leq k \leq k_n$.
In particular, 
this holds for $k=k_n$
and we have 
\[
\phi^\iota (w_{k_n}) = 
\phi^\iota (w_0) + k_n + O\Big(\frac{k_n^3}{n^2}\Big).
\]
The assertion follows
from the assumption 
\eqref{eq:cond-kn}
 on $k_n$.
\end{proof}

\begin{prop}\label{p:entering}
	We have
	\begin{equation}
	\label{eq:formula-entering}
	n
 \phi(w_{k_n}) = \phi^\iota(w_{k_n}) 
- 
  (1-a) \log (\pi /n)+ o(1).
 \end{equation}
Moreover, 
we have $W_{k_n}, W_{k_n+1}\in R_n$ for every $n$ sufficiently large.
\end{prop}

\begin{proof}
Let us start proving that
\begin{equation}\label{eq:estimate-psi-wkn}
n\psi  (w_{k_n}) = -\frac{1}{w_{k_n}} + o(1),
\qquad n\to \infty.
\end{equation}
Indeed, for all $k$
 we have
\[
n\psi
 (w_{k})
 =
n \frac{1}{2i\pi } \log \frac{w_k-\zeta^+}{w_k-\zeta^-}
= n \frac{1}{2i\pi } \log \frac{1-\frac{\zeta^+}{w_k}}{1-\frac{\zeta^-}{w_k}}
=n \frac{1}{2i\pi } 
\Big(
\log (1-\frac{\zeta^+}{w_k}) - \log (1-\frac{\zeta^-}{w_k}) \Big).
\]
By Lemma \ref{l:zeta}, we have
 $\zeta^+ -\zeta^- = \frac{2i\pi}{n}$. Moreover, 
 by Lemma \ref{l:phi-iota-kn} we also have
 $|1/ w_{k_n}|\sim k_n
 =o (n)$.
 Hence,
$\frac{\zeta^{\pm}}{w_{k_n}} = o(1)$ and
\[
n\psi
 (w_{k_n})
= n \frac{1}{2i\pi} \frac{\zeta^+ - \zeta^{-}}{w_{k_n}} +  n O(\frac{1}{ n^3 \cdot |w_{k_n}|^2})
=- \frac{1}{w_{k_n}} 
+ O\Big(\frac{k_n^2}{n^2}\Big)
=- \frac{1}{w_{k_n}} + o(1).
\]

Recall that $\phi = \chi \circ \psi$
and that $\phi^\iota$ satisfies \eqref{eq:phi-precise}.
 In order to establish
 \eqref{eq:formula-entering},
 we need to prove that
  \[\log \sin (\pi \psi (w_{k_n}) )=- \log (-w_{k_n}) + {\log (\pi/ n)} 
+  o(1)\]
  (observe that we do not need to do this estimate if $a=1$, as $\phi=\psi$ in that case).
By the 
estimate
\eqref{eq:estimate-psi-wkn}
for $\psi(w_{k_n})$
 and the facts that $w_{k_n}\sim 1/k_n$
and $\frac{k_n^3}{n^3} = o(\frac{1}{n})$ by the assumption \eqref{eq:cond-kn},
 we have
\[
\begin{aligned}
\log  \sin (\pi \psi (w_{k_n}) )
& =
\log  \sin(\pi  (- \frac{1}{n w_{k_n}} + o(\frac{1}{n})) )
 =  \log  (- \frac{\pi}{n w_{k_n}} + o(\frac{1}{n} ) + O(\frac{k_n^3}{n^3}))\\
& =  \log  (- \frac{\pi}{n w_{k_n}} + o(\frac{1}{n} ) )
=  \log  (- \frac{\pi}{n w_{k_n}} ) + \log (1 - w_{k_n} o(1) ).
\\
&=\log  (- \frac{\pi}{n w_{k_n}} ) + o(1)
 = - \log (-w_{k_n}) + \log (\pi/n) + o(1).
\end{aligned}
\]
This concludes the proof of \eqref{eq:formula-entering}.

\medskip

 Let us now show that
$W_{k_n} \in R_n$
for every $n$ sufficiently large.
By Lemma
\ref{l:phi-iota-kn} 
we have 
$w_{k_n}=-(1+o(1))/k_n$, hence $|\zeta_\pm|/|w_{k_n}|=O(k_n/n)=o(1)$.
Using the definition of $\psi$ we get
\[
\psi(w_{k_n})=\frac{1}{2i\pi}\log\frac{w_{k_n}-\zeta_+}{w_{k_n}-\zeta_-}
=\frac{k_n}{n}+o\!\left(\frac{k_n}{n}\right).
\]
Since $\chi(W)=W-\frac{1-a}{n}\log\sin(\pi W)$ satisfies $\chi(W)=W+O(\log(n/k_n)/n)$ on $R_n$,
it follows that $W_{k_n}:=\phi(w_{k_n})=\chi(\psi(w_{k_n}))\in R_n$ for $n$ large.

\medskip

Finally,
 from $w_{k_n+1}=f(w_{k_n})+\varepsilon_{k_n,n}^2$ and $f(w)=w+w^2+O(w^3)$ we have
$\frac{1}{w_{k_n+1}}=\frac{1}{w_{k_n}}-1+O(w_{k_n})$, hence
\[
\psi(w_{k_n+1})-\psi(w_{k_n})
=\frac{1}{n}+O\!\left(\frac{1}{n k_n}\right)=\frac{1}{n}+o\!\left(\frac{1}{n}\right).
\]
Therefore $W_{k_n+1}-W_{k_n}=\chi(\psi(w_{k_n+1}))-\chi(\psi(w_{k_n}))=\frac{1}{n}+o(1/n)$.
Since $W_{k_n}\sim k_n/n$
 lies at distance
at least $k_n/2n$ 
from 
the boundary of $R_n$, we also get $W_{k_n+1}\in R_n$.
This concludes the proof.
\end{proof}

\subsection{Passing through the eggbeater}\label{ss:passing}
As a result of the estimates of the previous section and the definitions of $\phi$
and $R_n$,
the element $w_{k_n}$ satisfies $\phi (w_{k_n})\in R_n$ as soon as $n$ is sufficiently large.
The goal of this section is to show that $\phi(w_k)$
stays in $R_n$ for every $k_n\leq k\leq n-k_n$, and to estimate the error made at every step with respect
to the autonomous system. This is the point where we use the control on the error terms 
in Section \ref{s:coordinates}.
We start with the following preliminary estimate.

\begin{lem}\label{l:prelim}
	For every $ k_n \leq j \leq n - k_n$, we have $W_j \in R_n$ and 
	$$W_j = W_{k_n} + \frac{j - k_n}{n}+ O\Big(\frac{j}{n^2}
\Big).
$$
\end{lem}

Observe that,
by 
the assumptions 
\eqref{eq:cond-kn} on $k_n$, we have
$\frac{\log (n/k_n)}{k_n^2}=o(1/n)$,
 hence errors of this magnitude 
can be  considered negligible in the proof below.

\begin{proof}
We argue by induction on $j$ in the range $k_n\le j\le n-k_n$.

\smallskip
\noindent\textbf{Base step.}
By Proposition~\ref{p:entering} we have $W_{k_n},W_{k_n+1}\in R_n$.
The estimate is trivial for $j=k_n$, and the proof of Proposition~\ref{p:entering} yields
\[
W_{k_n+1}=W_{k_n}+\frac1n+O\!\left(\frac1{n^2}\right),
\]
so it holds for $j=k_n+1$ as well.

\smallskip
\noindent\textbf{Induction step.}
Fix $J$ with $k_n\le J\le n-k_n-2$ and assume that
\begin{itemize}
\item $W_j\in R_n$ for all $k_n\le j\le J+1$, and
\item for all $k_n\le j\le J+1$,
\[
W_j = W_{k_n}+\frac{j-k_n}{n}+O\!\left(\frac{j}{n^2}
\right).
\]
\end{itemize}
We prove that 
(for every $n$ sufficiently large)
we have $W_{J+2}\in R_n$ and
 the estimate holds for $j=J+2$.
Observe that
we  first need to prove $W_{J+2}\in R_n$ by a rough one-step estimate, since
 Proposition~\ref{prop:estimate-central} is stated under the a priori assumption
 $W_{J+2}\in R_n$.
 
\smallskip
\noindent\emph{Step 1 (one--step invariance): $W_{J+2}\in R_n$.}
From the induction hypothesis we have
\[
\re (W_{J+1})=\re (W_{k_n})
+\frac{J+1-k_n}{n}+O\!\left(\frac{J+1}{n^2}
\right),
\]
hence for $n$ large we have
\[
{\rm dist}(W_{J+1},\partial R_n)\ge \frac{k_n}{20n}.
\]

Write $w_{J+2}=f(w_{J+1})+\varepsilon_{J+1,n}^2$.
Then
\[
w_{J+2}-w_{J+1}=w_{J+1}^2+O(w_{J+1}^3)+O(1/n^2).
\]
Since $W_{J+1}\in R_n$, Lemma~\ref{l:psi-phi}(2)
 gives $|w_{J+1}-\zeta_\pm|\ge c/n$
and Lemma~\ref{l:psi-phi}(1)
gives $|w_{J+1}|\lesssim 1/k_n$ on $R_n$
(we use here that  $|\cot(\pi W_{J+1})|=O(n/k_n)$). Therefore
\[
|w_{J+2}-w_{J+1}|\;\lesssim\; \frac1{k_n^2}+\frac1{n^2}.
\]

Denote as above $U_{J+1}:=\psi(w_{J+1})$.
 Differentiating the definition of $\psi$ yields
\[
\psi'(w)=\frac{1}{n(w-\zeta_+)(w-\zeta_-)}.
\]
Using $|w_{J+1}-\zeta_\pm|\ge c/n$, we get $|\psi'(w_{J+1})|\lesssim
\frac{1}{n\max(|w_{J+1}|,1/n)^2}$, hence
\[
|\psi'(w_{J+1})|\cdot |w_{J+2}-w_{J+1}|\;\lesssim\;\frac1n.
\]
A Taylor expansion gives $A_{J+1}=\psi(w_{J+2})-\psi(w_{J+1})=O(1/n)$.

Recall that we denote
 $W=\chi(U)$.
On $R_n$ we have $|\cot(\pi U)|\lesssim n/k_n$, hence
\[
\chi'(U)=1-\frac{1-a}{n}\pi\cot(\pi U)=1+O(1/k_n),
\]
so $\chi$ is $1+o(1)$--Lipschitz on $R_n$. Therefore
\[
|W_{J+2}-W_{J+1}|
=|\chi(U_{J+2})-\chi(U_{J+1})|
\le (1+o(1))|A_{J+1}|
=O(1/n).
\]
Since ${\rm dist}(W_{J+1},\partial R_n)\ge k_n/(20n)$,
 we conclude that $W_{J+2}\in R_n$
for every $n$ 
sufficiently 
large, as desired.

\smallskip
\noindent\emph{Step 2: estimate for $W_{J+2}$.}
Since
 $W_{J+1},W_{J+2}\in R_n$, we may apply
Proposition~\ref{prop:estimate-central} at time $J+1$ to get
(as observed above, the term
 $\frac{\log (n/k_n)}{k_n^2}$
is negligible with respect to $1/n$)
\[
W_{J+2}-W_{J+1}
=\frac1n+O\!\left(\frac1{n^2}\right),
\]
since $H_{J+1,n}(\cdot)=O(1/n^2)$ on $R_n$ (by Lemma \ref{l:psi-phi}(2).
Combining this with the induction hypothesis for $W_{J+1}$ yields
\[
W_{J+2}
= W_{k_n}+\frac{J+2-k_n}{n}+O\!\left(\frac{J+2}{n^2}\right).
\]
This concludes the induction step and the proof.
\end{proof}

\begin{prop}\label{p:middle}
	For every 
	$ k_n \leq j \leq n - k_n$
	 we have $W_j \in R_n$ 
	and 
	$$W_j =
	 \left( W_{k_n} + \frac{j-k_n}{n}+ \sum_{k=k_n}^{j-1}
	 H_{k,n}\Big(
	 -\frac{\pi}{n} \cot
	\Big(\frac{{(k+1)}\pi}{n}\Big)\Big) \right)+ 
	{o\Big(\frac{1}{n}\Big)}.$$
\end{prop}

A more precise error term
in the expression above would be 
	$O\Big(\frac{1}{n^{1+\alpha}}, \frac{\log (n/k_n)}{k_n^2}\Big)$. This is
	indeed $o(1/n)$ because of the assumptions \eqref{eq:cond-kn} on $k_n$.
An error $o(1/n)$ will be enough for our final estimate in Section \ref{ss:end}.

\begin{proof}
We refine the estimate of Lemma \ref{l:prelim}.
 We only need to prove the formula. We again work by induction and prove that
\[W_j = \left( W_{k_n} + \frac{j-k_n}{n}+ 
\sum_{k=k_n}^{j-1} H_{k,n}(-\frac{\pi}{n} \cot(\frac{{(k+1)}\pi}{n})) \right)+
{o(\frac{j}{n^2})}\]
for every $j$ as in the statement.
For $j= k_n$ the estimate is trivial
(the sum is empty in this case).
 For $k_n < J < n - k_n$,
by 
 Lemma \ref{l:prelim}  we can apply 
 Proposition \ref{prop:estimate-central}, which gives
\[
\begin{aligned}
W_{J+1}   = W_J + \tilde A_J
= &
\left( W_{k_n} + \frac{J-k_n}{n}+ \sum_{k=k_n}^{J-1}
 H_{k,n}(-\frac{\pi}{n} \cot(\frac{{(k+1)} \pi}{n})) \right)
 {+ o(\frac{J}{n^2})}
\\
& +
\left(
\frac{1}{n}
+
H_{J,n}(-\frac{\pi}{n}\cot(\pi W_{J+1})) 
{+o(\frac{1}{n^2})}
\right),
\end{aligned}
\]
where we  used the 
assumption
\eqref{eq:cond-kn} on
$k_n$
to bound with $o(1/n^2)$ the error term in Proposition \ref{prop:estimate-central}
(here, we may get a more precise error term
$O\Big(\frac{1}{n^{2+\alpha}}, \frac{\log (n/k_n)}{n\, k_n^2}\Big)$
leading to the more precise error term mentioned above).
It follows that
\[
\begin{aligned}
W_{J+1}= 
W_{k_n} 
& + \frac{J+1-k_n}{n}\\
&+ 
\left( \sum_{k=k_n}^{J-1} 
H_{k,n}(-\frac{\pi}{n} \cot(\frac{{ (k+1)}\pi}{n})) 
+
H_{J,n}(-\frac{\pi}{n}\cot(\pi W_{J+1}))
\right)
+ o\Big(\frac{J}{n^2}\Big).
\end{aligned}
\]
To conclude, we
need to prove that
\[H_{J,n}(-\frac{\pi}{n}\cot(\pi W_{J+1}))-
H_{J,n}(-\frac{\pi}{n}\cot(\frac{{(J+1)}\pi}{n}))= o(\frac{1}{n^2})\]
(observe that each  term $H_{k,n}$ is $O(1/n^2)$).
By the definition of $H_{k,n}$, it is enough to prove that
\begin{equation}\label{eq:cot-cot}
\cot(\pi W_{J+1})  - \cot(\frac{{ (J+1)}\pi}{n}) = o(1)
\end{equation}
By Lemma
\ref{l:prelim},
 we have
$W_{J+1}=  \frac{J+1}{n} + O(\frac{1}{n})$. Hence, we have
\[
|\cot(\pi W_{J+1})  - \cot(\frac{{(J+1)}\pi}{n}) |
\lesssim  \cot' (\frac{(J+1)\pi}{n}) \cdot O(\frac{1}{n})
\]
Since $\cot (\cdot)$ has
 a simple pole at $0$, its derivative has a double pole at the same point. Since $J\geq k_n$, we deduce that
  $| \cot' (\frac{(J+1)\pi}{n})| = O(\frac{n^2}{k_n^2})$. Hence,
 \[ |\cot(\pi W_{J+1})  - \cot(\frac{{(J+1)}\pi}{n}) |
= O(\frac{n}{k_n^2}) = o(1),
 \]
 where in the last step we used the
 assumption \eqref{eq:cond-kn} on $k_n$. 
\end{proof}

We conclude this part with the following estimate for the point $w_{n-k_n}$, which we will need
to initialize the next part (we will actually need only the bound
$w_{n-k_n}=O(1/k_n)$, but the proof for the precise expression is essentially the same).

\begin{lem}\label{l:size-exit}
We have
\[
w_{n-k_n}=\frac{1}{k_n}+O\!\left(\frac{1}{n}\right).
\]
In particular, $w_{n-k_n}=O(1/k_n)$.
\end{lem}

\begin{proof}
By Lemma~\ref{l:prelim} we have
\[
W_{n-k_n}=\frac{n-k_n}{n}+O\!\left(\frac{1}{n}\right)=1-\frac{k_n}{n}+O\!\left(\frac1n\right).
\]
Set $\eta_n:=1-W_{n-k_n}$, so that $\eta_n=\frac{k_n}{n}+O(1/n)$ and
$n\eta_n=k_n+O(1)$.
By Lemma \ref{l:psi-phi}(1), we have
\[
w_{n-k_n}=-\frac{\pi}{n}\cot(\pi W_{n-k_n})
+O\!\left(\frac1{n^2}\right)+O\!\left(\frac{\log(n/k_n)}{k_n^2}\right).
\]
Since $\cot(\pi(1-\eta))=-\cot(\pi\eta)=-(\pi\eta)^{-1}+O(\eta)$ as $\eta\to0$,
we obtain
\[
-\frac{\pi}{n}\cot(\pi W_{n-k_n})
=\frac{1}{n\eta_n}+O\!\left(\frac{\eta_n}{n}\right)
=\frac{1}{k_n}+O\!\left(\frac{1}{n}\right).
\]
This gives the assertion
(recall that, by
\eqref{eq:cond-kn}, 
$\frac{\log(n/k_n)}{k_n^2}$ is
negligible with respect to  $\frac{1}{n}$).
\end{proof}

\subsection{Exiting the eggbeater}\label{ss:exiting}
The next two estimates 
mirror those of 
 Proposition \ref{p:entering} and Lemma \ref{l:phi-iota-kn}.

\begin{prop}\label{p:exiting}
	We have
\[
n\bigl(\phi(w_{n-k_n})-1\bigr)
= \phi^o(w_{n-k_n}) 
 -
  (1-a)\log(\pi/n)
 + o(1).
\]
\end{prop}

\begin{proof}
Set $\eta_n:=\psi(w_{n-k_n})-1$. Since $\psi(w_{n-k_n})\in (0,1)+i(-1,1)$ and
$w_{n-k_n}=O(1/k_n)$ (by Lemma \ref{l:size-exit})
 a Taylor expansion of the logarithm in the definition of $\psi$
shows that
\[
n\eta_n = -\frac{1}{w_{n-k_n}}+o(1).
\]
Moreover, we have
\[
\sin(\pi\psi(w_{n-k_n}))=\sin(\pi(1+\eta_n))=-\sin(\pi\eta_n)
= -\pi\eta_n+o(1/n),
\]
hence
\[
\log\sin(\pi\psi(w_{n-k_n})) = \log
\Bigl(\frac{\pi}{n\,w_{n-k_n}}\Bigr)+o(1)
= -\log w_{n-k_n}+\log(\pi/n)+o(1).
\]
Recalling 
that $\phi=\chi\circ\psi$,
 we get
\[
\begin{aligned}
n(\phi(w_{n-k_n})-1)
&=n(\psi(w_{n-k_n})-1)
-
(1-a)\log\sin(\pi\psi(w_{n-k_n}))\\
& = -\frac{1}{w_{n-k_n}}
+
(1-a)\log w_{n-k_n}-(1-a)\log(\pi/n)+o(1).
\end{aligned}
\]
Using the asymptotic 
\eqref{eq:phi-precise}
for $\phi^o$
gives the assertion.
\end{proof}

\begin{lem}\label{l:new-exit}
We have
\[
\phi^o(w_n)=\phi^o(w_{n-k_n})+k_n+o(1).
\]
\end{lem}

\begin{proof}
Arguing exactly
 as in Lemma
\ref{l:phi-iota-kn},
 but starting at time $n-k_n$,
and using the estimate for $w_{n-k_n}$ given by Lemma
\ref{l:size-exit},
  we obtain
\begin{equation}\label{eq:w-exit}
w_{n-k_n+\ell}=\frac{1+o(1)}{k_n-\ell}
\qquad\text{for }0\le \ell\le k_n-1.
\end{equation}
Since $\phi^o\circ f=\phi^o+1$, $(\phi^o)'(w)=O(1/w^2)$ near $0$, and
$|\varepsilon_{j,n}|^2=O(1/n^2)$, for every $n-k_n\le j\le n-1$ we get
\[
\begin{aligned}
\phi^o(w_{j+1})
 =
\phi^o(f(w_j)+\varepsilon_{j,n}^2)
& =
\phi^o(f(w_j))
+
O\!\left(\frac{1}{n^2}\,\big|(\phi^o)'(f(w_j))\big|\right)\\
& =
\phi^o(w_j)+1+O\!\left(\frac{1}{n^2|w_{j+1}|^2}\right).
\end{aligned}\]
Therefore,
\[
\phi^o(w_n)-\phi^o(w_{n-k_n})
=
k_n
+
O\!\left(
\frac{1}{n^2}\sum_{\ell=0}^{k_n-1}\frac{1}{|w_{n-k_n+\ell+1}|^2}
\right).
\]
Using 
\eqref{eq:w-exit},
we obtain
\[
\phi^o(w_n)-\phi^o(w_{n-k_n})
=
k_n
+
O\!\left(
\frac{1}{n^2}\sum_{m=1}^{k_n} m^2
\right)
=
k_n+O\!\left(\frac{k_n^3}{n^2}\right)
=
k_n+o(1),
\]
where in the last step we use the assumption 
\eqref{eq:cond-kn} on $k_n$.
\end{proof}

\subsection{End of the proof of Theorem \ref{th:main}}\label{ss:end}

We can now conclude the proof of Theorem \ref{th:main}. 
Applying Proposition
\ref{p:middle}
 with $j=n-k_n$
and multiplying by $n$,
 we obtain
\[
nW_{n-k_n}
=
nW_{k_n}
+n-2k_n
+
n\sum_{m=k_n}^{\,n-k_n-1}
H_{m,n}\!\left(-\frac{\pi}{n}\cot\!\left(\frac{(m+1)\pi}{n}\right)\right)
+o(1).
\]
By Proposition
\ref{p:entering} 
and Lemma
\ref{l:phi-iota-kn},
we have
\[
nW_{k_n}
=
\phi^\iota(w_{k_n})-(1-a)\log(\pi/n)+o(1)
=
\phi^\iota(w_0)+k_n-(1-a)\log(\pi/n)+o(1).
\]
On the outgoing side, Proposition
\ref{p:exiting} gives
\[
\phi^o(w_{n-k_n})
=
nW_{n-k_n}-n+(1-a)\log(\pi/n)+o(1).
\]
Combining the previous identities, we get
\[
\phi^o(w_{n-k_n})
=
\phi^\iota(w_0)-k_n
+
n\sum_{m=k_n}^{\,n-k_n-1}
H_{m,n}\!\left(-\frac{\pi}{n}\cot\!\left(\frac{(m+1)\pi}{n}\right)\right)
+o(1).
\]
Finally, 
 applying Lemma \ref{l:new-exit} we obtain
\[
\phi^o(w_n)
=
\phi^\iota(w_0)
+
n\sum_{m=k_n}^{\,n-k_n-1}
H_{m,n}\!\left(-\frac{\pi}{n}\cot\!\left(\frac{(m+1)\pi}{n}\right)\right)
+o(1).
\]
We can
 also extend the sum above
  to all indices $0\le m\le n-1$, since the omitted boundary terms contribute $o(1)$ after multiplication by $n$.
  Indeed, for the boundary indices one has
\[
\begin{cases}
G\!\left(\frac{m+1}{n}\right)=O\!\left(\frac{(m+1)^2}{n^2}\right)
\quad\text{for }m<k_n,\\
G\!\left(\frac{m+1}{n}\right)=O\!\left(\frac{(n-m)^2}{n^2}\right)
\quad\text{for }n-k_n\le m\le n-1.
\end{cases}\]
Hence the omitted contribution is
\[
O\!\left(\frac1n\sum_{m<k_n}\frac{(m+1)^2}{n^2}\right)
+
O\!\left(\frac1n\sum_{n-k_n\le m\le n-1}\frac{(n-m)^2}{n^2}\right)
=
O\!\left(\frac{k_n^3}{n^3}\right)
=o(1),
\]  
  where in the last step we used the assumption \eqref{eq:cond-kn} on $k_n$.
  As a result, 
  we obtain
\[
\phi^o(w_n)
=
\phi^\iota(w_0)
+
n\sum_{m=0}^{n-1}
H_{m,n}\!\left(-\frac{\pi}{n}\cot\!\left(\frac{(m+1)\pi}{n}\right)\right)
+o(1).
\]

To conclude, we 
need to compute 
 the sum on the right hand side.
It follows from the definition \eqref{eq:Hkn} of $H_{k,n}$
(see Lemma \ref{l:zeta} for the definition of $\delta_k$)
that
\begin{align*}
	H_{k,n}(w)&=\frac{2\pi^2 \sigma_k}{2i\pi n^3}  \frac{\zeta^+-\zeta^-}{w^2-(\zeta^+ + \zeta^-)w+\zeta^+ \zeta^- } 
	=\frac{2\pi^2 \sigma_k}{2i\pi n^3}  \frac{\frac{2i\pi}{n}}{w^2- \frac{a\pi^2}{n^2} w + \frac{\pi^2}{n^2} + O(\frac{1}{n^3})} \\
	&=\frac{2\pi^2 \sigma_k}{n^4}    \frac{1}{w^2-\frac{a\pi^2}{n^2}w+ \frac{\pi^2}{n^2} + O(\frac{1}{n^3}) }.
\end{align*}
Therefore, we have
\begin{align*}
	H_{k,n}
	\Big( - \frac{\pi}{n}
	& \cot
	\Big(\frac{(k+1)\pi}{n}\Big)\Big)\\
	&=
		\frac{2\pi^2 \sigma_k}{n^4} 
	\left( \frac{\pi^2}{n^2} \cot^2\Big(\frac{(k+1)\pi}{n}\Big)
	 + \frac{a\pi^3}{n^3} \cot\Big(\frac{(k+1)\pi}{n}\Big) 
	  + \frac{\pi^2}{n^2} + O(\frac{1}{n^3})  \right)^{-1}  \\
	&= 	\frac{2 \sigma_k}{n^2} 
	\left( \cot^2\Big(\frac{(k+1)\pi}{n}\Big) + \frac{a\pi}{n} \cot\Big(\frac{(k+1)\pi}{n}\Big)  +1+ O(\frac{1}{n})  \right)^{-1}.
\end{align*}	
It follows that 
\begin{align*}
	n \sum_{k=0}^{n-1}  H_{k,n}\Big(-\frac{\pi}{n}
	& \cot\Big(\frac{(k+1)\pi}{n}\Big)\Big) \\
	&= 
		\frac{2}{n} \sum_{k=0}^{n-1}  \sigma_k
	\left(
	 \cot^2\Big(\frac{(k+1)\pi}{n}\Big) + 
	 \frac{a\pi}{n} 
	 \cot\Big(\frac{(k+1)\pi}{n}\Big) 
	  +1+ O(\frac{1}{n})  \right)^{-1} \\
	&=  \left(	\frac{1 }{n} \sum_{k=0}^{n-1} G(\frac{(k+1)}{n}) \sigma_k \right) + o(1).
\end{align*}
The proof is complete.

\begin{rem}
We observe
that $\int_0^1 G(x) dx=1$, so that if $\sigma_k\equiv\sigma$ independently 
of $k$, 
	we do find again $\frac{ \sigma}{n} \sum_{k=0}^{n-1} G(\frac{k+1}{n})   = \sigma + o(1)$ 
	(as a Riemann sum), from which we recover the 
	usual
	statement of the autonomous parabolic implosion.
\end{rem}

\section{Proof of Corollaries \ref{cor:propA}--\ref{cor:measurableskp}}\label{s:corollaries}

\begin{proof}[Proof of Corollary \ref{cor:propA}]
	Fix
	 $n \in \N$. For $j \geq 0$, it will be convenient to set $\eps_j := \eps_{j,n}= \frac{\pi}{2}\sqrt{p^{n^2+j}(z)}$, 
	where $\sqrt{\cdot}$ denotes the square root with positive real part.
	For $z \in \mathcal B_p$,
	 we have $p^k(z) = \frac{1}{k+O(\ln k)}$ as $k\to \infty$,
	  so that 
	\begin{equation*}
		\eps_j=\frac{\pi}{2n}- \frac{\pi}{2n} \frac{j}{2n^2}+O\left(\frac{\ln n}{n^3}\right).
	\end{equation*}
Denote
$N:=2n$ and $\sigma_j:=-\frac{j}{n}=-\frac{2j}{N}$, so that 
	\begin{equation*}
		\eps_j = \frac{\pi}{N}+ \frac{\pi}{N^2} \sigma_j + O\left(\frac{\ln N}{N^3}\right).
	\end{equation*}
	Let $(z_j, w_j)=F^{j}(p^{n^2}(z), w)$. Then, by Theorem \ref{th:main}
	 we have $w_N = L_{u}(w) + o(1)$, 
	where $u:=\frac{1}{N} \sum_{j=0}^{N-1} \sigma_j G(\frac{j+1}{N})$.
Hence,
	\begin{align*}
		u&=-\frac{1}{N} \sum_{j=0}^{N-1}
		 \frac{2j}{N} G(\frac{j+1}{N}) =-\int_0^1 4 x \sin^2(\pi x) dx + o(1) =- 1 + o(1).
	\end{align*}
	Therefore $w_N=L_{-1}(w)+o(1)$, and so $w_{2n+1}=L_0(w)+o(1)$, as required.
\end{proof}

\begin{proof}[Proof of Corollary \ref{cor:vivas}]
By Theorem
\ref{th:main}, 
it suffices to prove that
\[
u_n:=\frac1n\sum_{k=0}^{n-1}\sigma_{k,n}\,G\!\left(\frac{k+1}{n}\right)\longrightarrow 0,
\qquad 
n\to \infty.
\]
Observe that, since $G(1)=2\sin^2(\pi)=0$, the term corresponding to $k=n-1$ in the sum above vanishes.
Moreover, 
since $G(x)=G(1-x)$ for every $x\in[0,1]$, for every $0\le k\le n-2$ we have
\[
G\!\left(\frac{k+1}{n}\right)
=
G\!\left(1-\frac{k+1}{n}\right)
=
G\!\left(\frac{n-1-k}{n}\right)
=
G\!\left(\frac{(n-2-k)+1}{n}\right).
\]
Therefore, we obtain
\[
u_n
=
\frac{1}{2n}\sum_{k=0}^{n-2}
\bigl(\sigma_{k,n}+\sigma_{n-2-k,n}\bigr)
\,G\!\left(\frac{k+1}{n}\right).
\]
Since $G$ is bounded on $[0,1]$,
using \eqref{eq:symm-sigma}
we obtain
\[
|u_n|
\le
\frac{\|G\|_\infty}{2n}
\sum_{k=0}^{n-2}
\left|\sigma_{k,n}+\sigma_{n-2-k,n}\right|
=
O\!\left(\frac1n\right).
\]
Hence $u_n\to 0$, and the conclusion follows from Theorem
\ref{th:main}.
\end{proof}

\begin{proof}[Proof of Corollary \ref{cor:random}]
By Theorem
\ref{th:main}, 
it suffices to prove that the phases
\[
u_n:=\frac1n\sum_{k=0}^{n-1}\sigma_{k,n}\,G\Big(\frac{k+1}{n}\Big)
\]
converge almost surely to $u$,
where $G(x)=2\sin^2(\pi x)$ is as in Theorem \ref{th:main}.

Set $b_j:=\sigma_{j-1}$ for $j\ge 1$. Then we have
\[
u_n=\frac1n\sum_{j=1}^{n} b_j\,G\!\left(\frac{j}{n}\right)
\quad
\text{ and }
\quad
\frac1n\sum_{j=1}^{n} b_j
=
\frac1n\sum_{j=0}^{n-1}\sigma_j
\longrightarrow u
\qquad\text{almost surely.}
\]
Therefore, Lemma
\ref{lem:weighted-cesaro}
 applied to the sequence $(b_j)$ and the function $g=G$
 gives
\[
u_n \longrightarrow u\int_0^1 G(x)\,dx
\qquad\text{almost surely.}
\]
Since
$\int_0^1 2\sin^2(\pi x)\,dx=1$,
we obtain $u_n\to u$ almost surely. Hence
we have
$w_n^{(n)}\to L_u(w_0)$
almost surely, for every $w_0\in B_f$. This concludes the proof.
\end{proof}

The following lemma is an elementary consequence of the discrete
integration by parts. We give a proof
for the reader's convenience. It will also
 be used in
the proof of Corollary \ref{cor:skewproduct}, see 
Proposition \ref{p:FJ}.
The assumption on the derivative of $g$ is not essential, 
but the proof is easier assuming it.

\begin{lem}\label{lem:weighted-cesaro}
Let $(b_k)_{k\ge1}$ be a bounded sequence such that
\[
\frac1n\sum_{k=1}^{n} b_k \longrightarrow L\in  \mathbb C \qquad (n\to\infty).
\]
Then for every $g\in C^1([0,1])$ one has
\[
\frac1n\sum_{k=1}^{n} b_k\,g
\Big(\frac{k}{n}\Big)
\longrightarrow
L\int_0^1 g(x)\,dx
\qquad (n\to\infty).
\]
\end{lem}

\begin{proof}
Set
$S_0:=0$
 and $S_j:=\sum_{k=1}^{j} b_k$ for $j\ge1$.
The assumption implies $S_j/j\to L$, hence $S_j=Lj+o(j)$.
As
$b_k=S_k-S_{k-1}$, a summation by parts gives
\[
\sum_{k=1}^{n} b_k\,g
\Big(\frac{k}{n}\Big)
=
S_n\,
g(1)-\sum_{k=1}^{n-1} S_k
\cdot \Big(g
\Big(\frac{k+1}{n}\Big)-g
\Big(\frac{k}{n}\Big)\Big),
\]
which gives
\[
\frac1n
\sum_{k=1}^{n} b_k\,g
\Big(\frac{k}{n}\Big)
=
L\,g(1)+o(1)
-\frac1n
\sum_{k=1}^{n-1} S_k\Big(g
\Big(\frac{k+1}{n}\Big)-g
\Big(\frac{k}{n}\Big)\Big),
\]
using that $S_n/n\to L$.
By the mean value theorem, we have
\[
g \Big(\frac{k+1}{n}\Big)-
g
\Big(\frac{k}{n}\Big)=\frac1n\,g'(\xi_{k,n})
\quad\text{for some }\xi_{k,n}\in(k/n,(k+1)/n),
\]
hence
\[
\frac1n\sum_{k=1}^{n-1} S_k\Big(g\Big(\frac{k+1}{n}\Big)-
g \Big(\frac{k}{n}\Big)\Big)
=
\frac{1}{n^2}\sum_{k=1}^{n-1} S_k\,g'(\xi_{k,n}).
\]
Using 
that $S_k=Lk+o(k)$ and
the boundedness of $g'$,
 we get
\[
\frac{1}{n^2}\sum_{k=1}^{n-1} S_k\,g'(\xi_{k,n})
=
\frac{L}{n^2}\sum_{k=1}^{n-1} k\,g'(\xi_{k,n})+o(1).
\]
Moreover,
\[
\frac{1}{n^2}\sum_{k=1}^{n-1} k\,g'(\xi_{k,n})
=
\frac1n\sum_{k=1}^{n-1}\frac{k}{n}\,g'(\xi_{k,n})
\longrightarrow
\int_0^1 x\,g'(x)\,dx,
\]
since $\xi_{k,n}\to k/n$ uniformly in $k$.
Therefore
\begin{equation}\label{eq:end-parts}
\frac1n\sum_{k=1}^{n} b_k\,g
\Big(\frac{k}{n}\Big)
\longrightarrow
L\,g(1)-L\int_0^1 x\,g'(x)\,dx.
\end{equation}
Finally,
an
 integration by parts gives
\[
g(1)-\int_0^1 x\,g'(x)\,dx=\int_0^1 g(x)\,dx.
\]
Hence, the limit in
 \eqref{eq:end-parts}
 is equal
 to $L\int_0^1 g(x)\,dx$. The assertion follows.
\end{proof}

\begin{proof}[Proof of Corollary \ref{cor:measurableskp}]
	 For a $\mu$-generic $z$,
	  we apply Corollary
	   \ref{cor:random} 
	   to the bounded sequence
$\sigma_k:=\sigma(T^k z)$.
Then, by Birkhoff's ergodic theorem, we have
\[
\frac1n\sum_{k=0}^{n-1}\sigma(T^k z)\longrightarrow \int_\Omega \sigma\, d\mu.
\]
Hence,
 the corresponding fiber dynamics converges to $L_u$, with
  $u =\int_\Omega \sigma\, d\mu$
   as desired.
\end{proof}

\section{Proof of Corollary \ref{cor:skewproduct}}
\label{s:cor-new}

Let $F$ be an endomorphism of $\mathbb P^2=\mathbb P^2 (\mathbb C)$ 
of algebraic degree $d\geq 2$
which is fibered over a rational map $p$ 
of degree $d$
on $\mathbb P^1= \mathbb P^1(\mathbb C)$,
 i.e.,
such that there exists a dominant rational map
$\pi\colon \mathbb P^2\dashrightarrow \mathbb P^1$
such that
\[
\pi \circ F = p\circ \pi.
\]
More explicitly,
after choosing homogeneous coordinates such that 
$\pi([z_1, z_2, w]) = [z_1, z_2]$,
 any such endomorphism can be written 
in the form
\[
F([z_1, z_2, w])
=
[p_1(z_1,z_2), p_2 (z_1, z_2), Q(z_1, z_2, w)]
\]
where $p=[p_1, p_2]$ is the expression of 
the rational map $p$ in homogeneous coordinates.
Geometrically, 
the map
$F$
sends each fiber of $\pi$
to another fiber, and the induced dynamics on the base $\mathbb P^1$, which parametrizes the
fibers,
 is given by 
$p$.
Recall from
\cite{jonsson1999dynamics, DT21}
 that the Julia set 
 $J(F)$
  and the equilibrium measure $\mu_F$
of $F$ 
admit
 a decomposition and a disintegration of the form
\[
J(F) = \bigcup_{z\in J(p)} J_z
\text{ and }
\mu_F = \int \mu_z d\mu_p (z)
\]
where $J_z\subset \pi^{-1} (z)$
and $\mu_z$ is a 
probability measure on
 $\pi^{-1} (z)$.

\begin{prop}\label{p:FJ}
Let $F: \ptwo \to \ptwo$ be a
holomorphic
	 endomorphism 
which, in some affine chart,
can be written in the form
	$F(z,w) = (p(z), q(w))$,
	where $p$ is a rational map and $q$ is a polynomial map of the form $q(w)=w+w^2+O(w^3)$.
	Let $F_n$ be a sequence of 
holomorphic endomorphisms of $\mathbb P^2$	
 which, in the same affine chart, have  the form
	$$F_n(z,w) = F(z,w) + \left(0,  \left( \frac{\pi}{n} + \frac{a(z)}{n^2} \right)^2\right).$$
Then,
for every ergodic $p$-invariant probability measure $\nu$
and
for $\nu$-a.e.\,
$z\in J(p)$
we have
\[
\liminf_{n\to\infty} J_z(F_n)\ \supset\ J_{\rm Lav}\!\left(q, \frac{1}{\pi}\int a\,d\nu\right),
\]
\end{prop}

As in \cite{lavaurs1989systemes},
given $u\in\mathbb C$,
the 
\emph{Julia-Lavaurs set} $J_{\rm Lav} (q,u)$
is defined as
\[
J_{\mathrm{Lav}}(q,u)
:=
\overline{\Bigl\{\,w \;:\; \exists m\ge 0 \text{ such that }
L_u^{\circ m}(w)\text{ is defined and }L_u^{\circ m}(w)\in J(q)\Bigr\}}.
\]
Observe that we have $J(q) \subset J_{\rm Lav}(q,u)$ for every $u\in\mathbb C$ and
$\cup_{u\in \mathbb C}J_{\rm Lav}(q,u)
= J \cup \mathcal B_q$, see for instance \cite{lavaurs1989systemes,Douady94Julia}.

\begin{proof}[Proof of Proposition \ref{p:FJ}]
For $z\in J_p$ 
and $n\ge1$, the fiber dynamics of $F_n$ over $z$ is
described by
 the non-autonomous
iteration
\[
w_{k+1}=q(w_k)+\varepsilon_{k,n}(z)^2,\qquad
\varepsilon_{k,n}(z):=\frac{\pi}{n}+\frac{a(p^k(z))}{n^2}
\qquad 0\le k\le n-1.
\]
Thus,
 it fits the framework of Theorem~\ref{th:main} 
 with
\[
\sigma_{k,n}(z)=\frac{1}{\pi}\,a(p^k(z)).
\]
Fix an ergodic $p$-invariant measure $\nu$ on $J_p$
 and let $z$ be $\nu$-generic.
Then, by Birkhoff's ergodic theorem, we have
\[
\frac{1}{n}\sum_{k=0}^{n-1} a(p^k(z))\ \longrightarrow\ \int a\,d\nu.
\]
Since $\int_0^1 G(t)\,dt=1$,
Lemma~\ref{lem:weighted-cesaro}
gives
\[
\frac1n\sum_{k=0}^{n-1}\sigma_{k,n}(z)\,
G\Big(\frac{k+1}{n}\Big)
=\frac{1}{\pi n}\sum_{k=0}^{n-1} a(p^k(z))\,G
\Big(\frac{k+1}{n}\Big)
\ \longrightarrow\ \frac{1}{\pi}\int a\,d\nu.
\]
Hence, by Theorem~\ref{th:main}, 
the
(non-autonomous) fiber dynamics
of $F_n$
 over such $z$ converges to the Lavaurs map
of phase $u=\frac{1}{\pi}\int a\,d\nu$. In particular,
we have
\[
\liminf_{n\to\infty} J_z(F_n)\ \supset\ J_{\rm Lav}(q,u).
\]
This inclusion is proved as in the one-dimensional case \cite{lavaurs1989systemes}
and is a consequence of the lower semi continuity of the fibered Julia sets,
 which in turn is a consequence
of the continuity 
(with respect to the map)
of the corresponding  measures $\mu_z$
\cite{jonsson1999dynamics, DT21}.
\end{proof}

Corollary
\ref{cor:skewproduct}
will follow combining the above proposition and the
following lemma.

\begin{lem}
\label{lem:rotation-interior}
Let $p\colon \mathbb P^1
\to \mathbb P^1$ be a rational map
of degree $d\geq 2$
and let $a\colon J(p)\to\C$
 be H\"older continuous. 
Assume that,
 for every $v\in \mathbb C\setminus \{0\}$, 
 the function $z\mapsto v\cdot a(z)$
  is not cohomologous to a constant on $J_p$,
  where $\cdot$ denotes the standard inner product on $\R^2 \simeq \C$.
  Then the set
\[ W:=\left\{\int_{J_p} a\,d\nu:\ \nu \ \text{ergodic $p$-invariant and }
{\rm supp}(\nu)= J_p\right\}\]
has nonempty interior in $\C$.
\end{lem}

\begin{proof} 
   For $\tau=(\tau_1,\tau_2)\in\R^2$ consider the real H\"older continuous
   potential \[ \varphi_\tau:=\tau_1\,\re (a)+\tau_2\,\im (a). \]
For every $\tau$ 
sufficiently small,
the potential $\varphi_\tau$ admits
   a unique equilibrium state $\nu_\tau$ for $\varphi_\tau$,
see for instance 
 \cite{BD23, DPU96}. 
 Moreover, $\nu_\tau$ is ergodic and has full support.
    Let $P(\tau):=P(\varphi_\tau)$ be the topological pressure.
    Again for $\tau$ sufficiently small,
  by \cite{BD24}
     the function $\tau\mapsto P(\tau)$
is real-analytic 
     and satisfies
     \[ \nabla P(\tau)=\left(\int \re(a)\,d\nu_\tau,\ \int \im(a)\,d\nu_\tau\right). \]
We observe here that the assumption {\bf (A)} in \cite{BD23,BD24} is not needed in our case, since 
     in dimension $1$ all periodic critical points are outside of the Julia set.
The assumption on $a$ implies that
$\det \mathrm{Hess}\,P(0)\neq 0$. 
To see this,  for
 $v\in\mathbb C\sim \R^2$,
define the function $g_v (z):=v\cdot a (z)$.
Then $g_v$ is H\"older-continuous and we have
\[
v^\top\,\mathrm{Hess}\,P(0)\,v = \sigma^2(g_v)\ge 0,
\quad
\text{ where} \quad
\sigma^2(g_v):= \int 
g_v^2 \, d\nu_0 + 2 \sum_{j=1}^\infty \int g_\nu (g_v \circ p^j )\,  d\nu_0\]
 denotes the asymptotic variance of $g_v$ with respect to $\nu_0$,
 see for instance 
\cite[Section 4]{PP90}
and \cite[Lemma 5.7]{BD24}.
If $\mathrm{Hess}\,P(0)$ were not positive definite, there would exist $v\neq 0$ with
$v^\top\mathrm{Hess}\,P(0)v=0$, hence $\sigma^2(g_v)=0$.
But,
again by \cite{PP90, BD24},
this implies
 that $g_v$ is cohomologous to a constant.
Therefore $\mathrm{Hess}\,P(0)$ is positive definite, 
hence invertible.

 In particular, by the inverse function theorem,
  $\nabla P$ is a local diffeomorphism near $0\in \mathbb \R^2$.
  Therefore,
   its
  image contains an open subset, as desired.
         \end{proof}

We can now prove Corollary
 \ref{cor:skewproduct}.

\begin{proof}[Proof of Corollary \ref{cor:skewproduct}]
We first observe that the assumption on $a$ in Corollary
\ref{cor:skewproduct} implies that of
 Lemma \ref{lem:rotation-interior}.

Indeed, fix $v\neq 0$ and suppose by contradiction that
$v\cdot a$ is cohomologous to a constant, i.e., that
the function
 $z\mapsto v\cdot a(z)=c+h-h\circ p$
 for some continuous function $h\colon J_p \to \mathbb C$.
 This implies that, 
for  $i=1,2,3$ we have
\[
v\cdot A_i
=\frac1{m_i}\sum_{j=0}^{m_i-1} v\cdot a(p^j (z^{(i)}))
=c.
\]
Thus the three points $A_1,A_2,A_3$
lie on the affine real line
\[
\{z\in\C:\ v\cdot z=c\},
\]
contradicting the assumption that they are not collinear.

By the above, we can apply Lemma~\ref{lem:rotation-interior}. 
Hence,
the set of phases
\[
W= \left\{\frac{1}{\pi}\int a\,d\nu:\ \nu \text{ ergodic, }{\rm supp}(\nu)=J_p\right\}
\]
contains a nonempty open set $\Omega\subset\C$. 
The Julia--Lavaurs set $J_{\rm Lav}(q,u)$ depends non-trivially 
on
 $u$
as the dependence is a translation in Fatou coordinates.
Hence,
 $\bigcup_{u\in\Omega}J_{\rm Lav}(q,u)$ contains a nonempty open subset $U\subset\C$.

The above,
Proposition \ref{p:FJ}, and the fact that the measures
in the definition of $W$ above have full support in $J(p)$
imply
 that every $z\in J_p$ can be approximated by
a sequence $z_m\in J_p$ with 
the property that, for every $m$, we have
\[
U \subset 
\liminf_{n\to\infty} J_{z_m}(F^n).
\]
As $J_{z_m} (F^n)\subset J(F^n)$
for every $n$ and $m$,
 the assertion follows from a diagonal argument.
\end{proof}

\printbibliography

\end{document}